\documentclass[12pt]{amsart}

\newtheorem{theorem}{Theorem}[section]

\newtheorem{lemma}[theorem]{Lemma}

\theoremstyle{definition}

\newtheorem{definition}[theorem]{Definition}

\newtheorem{remark}[theorem]{Remark}

\newtheorem{example}[theorem]{Example}

\theoremstyle{parrafo}

\begin{document}

\title[]{Variations on the Boman covering lemma}

\author{J. M. Aldaz}
\address{Instituto de Ciencias Matem\'aticas (CSIC-UAM-UC3M-UCM) and Departamento de 
Matem\'aticas,
Universidad  Aut\'onoma de Madrid, Cantoblanco 28049, Madrid, Spain.}
\email{jesus.munarriz@uam.es}
\email{jesus.munarriz@icmat.es}


\thanks{The author was partially supported by Grant MTM2015-65792-P of the
MINECO of Spain, and also by by ICMAT Severo Ochoa project SEV-2015-0554 (MINECO)}







\begin{abstract} We explore some variants of the Boman covering lemma, and  their relationship to the
boundedness properties of the maximal operator. Let $1 < p < \infty$ and let $q$ be its conjugate exponent.
We prove that the strong type
$(q,q)$ of the uncentered maximal operator, by itself, implies 
 certain generalizations of the Boman covering lemma for the exponent $p$, and in turn, these
 generalizations  entail the
 weak type 
$(q,q)$ of the centered maximal operator.  We show by example that it
is possible  for the uncentered maximal operator to be unbounded for all $1 < s < \infty$, while the conclusion
of the lemma holds for every $1 < p < \infty$; thus,  the latter condition is much weaker. 
Also, the boundedness of the centered maximal
operator  entails weak versions of the lemma.
\end{abstract}


\maketitle


\markboth{J. M. Aldaz}{The Boman Covering Lemma}

\section {Introduction}

Often called ``a useful lemma" or a ``rather well-known lemma", the Boman covering lemma appears for the first
time in \cite{Bom}, but  \cite{Boj} seems to be the standard reference  (I have not have access to  \cite{Bom}, a technical report).
 Hence,  I will refer to
 \cite{Boj}  in what follows. 
 
 Closely connected to what nowadays is called
 the Boman chain condition, the lemma has been used to extend inequalities
 of Sobolev and Poincar\'e type to a wide class of domains  
 (both in euclidean spaces and in several kinds of metric measure spaces)  to explore the fine
 properties of functions belonging to different function spaces, and even to seemingly unrelated areas, 
 such as  complex dynamics. Among the many papers
 that could be cited here, we mention 
 \cite{IwNo},  \cite{Ch},  \cite{Ch1},  \cite{ChWh},  \cite{IhVa}, \cite{HeKo}, \cite{KoShSt}.

In view of its usefulness it seems advisable, first, to name the lemma after its author, and
second, to explore possible variations and generalizations, studying their relationships  
with the  boundedness properties of maximal operators, both centered and
uncentered.

Given $1 < p < \infty$ and  Lebesgue measure on $\mathbb{R}^d$,
the Boman covering Lemma tells us
that the overlap of cubes (with sides parallel to the coordinate axes) 
when we increase their radii, grows in a controlled manner
in the $L^p$ norm (see Lemma \ref{Bojarski} below for the precise
statement). The result is an immediate consequence of H\"older's inequality and
the boundedness of the uncentered maximal operator.  Of course, the same proof works for metric measure spaces
with a doubling measure. 

Here we study alternative formulations  of the lemma where doubling is not assumed.
 After some fairly standard definitions in Section 2, 
Section 3 considers variants of Boman's lemma.  A normalized version is presented 
in Theorem \ref{normalizedBoman}, and  in 
Theorem \ref{generalizedBoman}, 
a generalization  where the smaller sets need not be balls, but can
be any measurable sets with positive measure.   
In both cases, the usual proof works without the doubling hypothesis, and the
result is obtained solely from the boundedness of the uncentered maximal operator $M^u$. This is interesting, 
since there are known boundedness results about $M^u$ which do not require doubling, often using instead
strong uniformity assumptions on the measure $\mu$ (see \cite{Io}) or weaker uniform assumptions on $\mu$ but
strong assumptions on the space (cf. \cite{Al1}). Also, assuming only the boundedness of the centered maximal operator $M$,
it is possible to prove a weak version of the lemma, valid only for some sufficiently large expansions, which depend
on the specific sequence of balls being considered, cf. Theorem \ref{centeredBoman}.

Let $p$ and $q$ be dual exponents. Since from the
conclusion of   Theorem \ref{generalizedBoman}
for the exponent $q$,  one can deduce the weak 
 type $(p,p)$ of the centered maximal operator $M$ (cf. Theorem \ref{M}) it is natural to ask what are the relative 
 strenghts of these properties. We prove that
the conclusion of Theorem \ref{normalizedBoman} is strictly weaker than  the strong type $(p,p)$ of $M^u$,
cf. Theorem \ref{overlapnotdoubling}. In the example considered there, 
$M^u$ is unbounded on $L^s$ for all $1 \le s < \infty$, while the 
conclusion of Theorem \ref{normalizedBoman} 
holds for all $1 \le q < \infty$.

Thus,  in principle it should be possible, in cases where the  boundedness
of $M^u$ either fails or is not known,  to obtain boundedness results for the centered maximal
function by directly proving the conclusions of Theorems \ref{normalizedBoman} or 
\ref{generalizedBoman}. 

I am indebeted to an anonymous referee for pointing out relevant bibliography. 

\section {Definitions and notation} 

In this section we introduce the basic definitions. Additional ones will be
presented when needed. We will use $B(x,r) := \{y\in X: d(x,y) < r\}$ to denote open balls, 
$\overline{B(x,r)}$ to denote their topological closures, and 
$B^{cl}(x,r) := \{y\in X: d(x,y) \le r\}$ to refer to metrically closed balls (consider
$\overline{B(0,1)}$ in $\mathbb{Z}$ to see the difference).  It is always assumed that measures are not
identically 0, and radii $r$ are always taken to be $> 0$.

Sometimes results are valid both for collections of open balls, for collections of closed balls, and
for mixed collections. To indicate this, we use the notation $U(x,r)$ for balls when we do not want to
specify whether they are open or closed. Thus,  $U(x,r)$ stands either for $B(x,r)$ or for
$B^{cl}(x,r)$. However, when dealing with contractions or expansions, we take $U(x,r)$
and $U(x,t r)$, $t > 0$, to be both open or both closed. We do  this since sometimes we allow $t=1$
in the statements of results (but the convention could be relaxed if it is assumed that $t\ne 1$).

\begin{definition} A Borel measure is   {\em $\tau$-additive} or {\em $\tau$-smooth}, if for every
collection  $\{U_\alpha : \alpha \in \Lambda\}$
 of  open sets, 
$$
\mu (\cup_\alpha U_\alpha) = \sup_{\mathcal{F}} \mu(\cup_{i=1}^nU_{\alpha_i}),
$$
 where the supremum is taken over all finite subcollections $\mathcal{F} = \{U_{\alpha_1}, \dots, U_{\alpha_n} \}$
of  $\{U_\alpha : \alpha \in \Lambda\}$.
 We say that $(X, d, \mu)$ is a {\em metric measure space} if
$\mu$ is a  $\tau$-additive  Borel measure on the metric space $(X, d)$, such that $\mu$ assigns finite measure
to bounded Borel sets. 
\end{definition} 

From now on we always assume that measures are locally finite (finite on bounded sets). For motivation regarding the definition of metric measure spaces using $\tau$-additivity,
cf. \cite{Al3}. Note that
in separable metric spaces all Borel measures are  $\tau$-additive, and the same happens
with  all Radon measures in arbitrary metric spaces.

\begin{definition}\label{maxfun} Let $(X, d, \mu)$ be a metric measure space and let $g$ be  a locally integrable function 
on $X$. For each  $x\in X$, the centered Hardy-Littlewood maximal operator $M_{\mu}$ 
 is given by
\begin{equation}\label{HLMFc}
M_{\mu} g(x) := \sup _{\{r >0 :  \ 0 < \mu (B^{cl}(x, r))\}}  \frac{1}{\mu
(B^{cl}(x, r))} \int _{B^{cl}(x, r)}  |g| d\mu.
\end{equation}
The uncentered Hardy-Littlewood maximal operator $M^u_{\mu}$ is defined via
\begin{equation}\label{HLMFu}
M^u_{\mu} g(x) := \sup _{\{r > 0, y\in X :  \  d(x,y) < r \mbox{ and } \mu B(y, r)^{cl}  > 0\}} 
\frac{1}{\mu
(B(y, r)^{cl})} \int _{B(y, r)^{cl}}  |g (u)| d\mu (u).
\end{equation}
When the supremum is taken over radii $r > R$, for some fixed $R > 0$, we obtain the following localized maximal operators:  
\begin{equation}\label{HLMFcR}
M_{\{ R < r\}, \mu} \  g(x) := \sup _{\{r >  R : \ 0 < \mu (B^{cl}(x, r))\}}  \frac{1}{\mu
(B^{cl}(x, r))} \int _{B^{cl}(x, r)}  |g| d\mu.
\end{equation}
In the uncentered case, the definition of  $M^u_{\{ R < r\}, \mu} $ is entirely analogous.
\end{definition}

Sometimes it is convenient to  use open balls in the definition,
instead of closed balls. This does not change the values of the Hardy-Littlewood maximal operators. 
Hence, we could also use $U(x,r)$ and take the supremum over both open and closed balls.
The reason why we define their localized versions using $r > R$ instead of
$r\ge R$, is so that the values of  $M_{\{ R < r\}, \mu} $  and  $M^u_{\{ R < r\}, \mu} $  do not change either, when
we use open or closed balls.

If we  consider only one measure $\mu$,  we often omit it from
the notation, writing, for instance, $M$ instead of  $M_{ \mu} $.

\begin{definition} A Borel measure $\mu$ on $(X,d)$ is {\em doubling}  if there exists a 
$C \ge 1 $ such that for all $r>0 $ and all $x\in X$, $\mu (B(x, 2 r)) \le C\mu(B(x,r)) < \infty$. 
\end{definition}

\begin{definition} \label{sub} Let $V, W$ be vector spaces over $\mathbb{R}$ or $\mathbb{C}$,
 let $C \subset V$ be a  cone and let $C^\prime \subset W$ be an ordered cone. We say that $T:C \to C^\prime$ is {\em sublinear} if for all $u, v \in C$ and
all $t > 0$, $T(t u + v) \le t T(u) + T(v)$.
\end{definition}

\begin{definition} \label{weak} Given $p$ with $1 \le p < \infty$,
a sublinear operator $T$ satisfies a
weak type $(p,p)$ inequality if there exists a constant $c > 0$ such that
\begin{equation}\label{weaktypep}
\mu (\{T g > \alpha\}) \le \left(\frac{c \|g\|_{L^p(\mu)}}{\alpha}\right)^p,
\end{equation}
where $c=c(p,  \mu)$ depends neither on $g\in L^p (\mu)$
nor on $\alpha > 0$. The lowest constant $c$ that satisfies the preceding
inequality is denoted by $\|T\|_{L^p\to L^{p, \infty}}$.
Likewise, if there exists a constant $c > 0$ such that
\begin{equation}\label{strongtypep}
\|T g \|_{L^p(\mu)}  \le c \|g\|_{L^p(\mu)},
\end{equation}
we say that $T$ satisfies a
strong type $(p,p)$ inequality.  The lowest such constant (the operator norm) 
 is denoted by $\|T\|_{L^p\to L^{p}}$.
\end{definition}

\section{Variations on the Boman covering lemma} 

Given $1 < p < \infty$, denote by  $q = p/(p-1)$ its conjugate exponent. Next we state the Boman covering lemma 
as presented in \cite[Lemma 4.2]{Boj}. 

\begin{lemma}  \label{Bojarski} {\bf (Boman)}  Let $F = \{Q_\alpha\}$  be an arbitrary family of cubes in 
$\mathbb{R}^d$ (indexed by some parameter $\alpha \in  I$). 
Assume that for each $Q_\alpha$ we are given a non-negative number $a_\alpha$. 
Then for $1 < p < \infty$  and  $ B \ge 1$ we have
$$
\left\|\sum_\alpha a_\alpha \mathbf{1}_{B Q_\alpha}
\right\|_{L^p} 
\le
D
\left\|\sum_\alpha a_\alpha\mathbf{1}_{Q_\alpha}
\right\|_{L^p} 
$$
where the constant $D$ depends on $d, p$ and $B$ only.
\end{lemma} 

In this section we explore different versions of  Lemma \ref{Bojarski}.
 The next result is a restatement of the lemma in a ``normalized" manner,
so while the proof  is the same (an exercise given two hints, according to \cite[p. 67]{HeKo})  the doubling hypothesis is avoided, 
and the result follows directly
from the  bounds for the uncentered maximal operator (we mention that under stronger
doubling hypotheses, a version of Boman's covering lemma for open sets in
$\mathbb{R}^d$ appears in
\cite[Lemma 2.3]{Bu}). 

Unlike Lemma \ref{Bojarski},
where the expansion factor $B$ is the same for every ball, the following variant allows different
expansion factors $t_n$ for different balls. While  the next result is stated for finite or countable
collections of balls only, rather than arbitrary collections,  we point out  that if the
collection of balls (all of them of positive measure)  is uncountable, and $a_\alpha > 0$ for an uncountable set of indices,
then the right hand side of (\ref{overlap}) is always infinite, and the inequality is trivially satisfied
(more details in remark \ref{count} below)
so there is no real loss of generality in considering only countable (or even finite) collections of balls.

Recall that $U(x,r)$ denotes a ball of center $x$ and radius $r$, without specifying whether it is open or closed
(but contractions or expansions of balls are assumed to have the same character as the original ones). 

\begin{theorem}  \label{normalizedBoman}  Let $(X, d, \mu)$ be a
metric measure space, fix $1 < p < \infty$, and set $q = p/(p-1)$.  Let $\{w_n\}_{n \in I}$ be a 
finite or countably infinite  sequence of weights, so $w_n > 0$,  let
$\{t_n\}_{n \in I}$ be a corresponding
finite or countably infinite  sequence of dilations, so $t_n \ge 1$,
and let $\{U(x_n, r_n)\}_{n \in I}$ be a sequence of balls with 
$\mu U(x_n,r_n) > 0$ for $1 \le n <N$. Then
\begin{equation}\label{overlap}
\left\| \sum_{n \in I}\frac{w_n \mathbf{1}_{U(x_{n},  t_n r_{n})} }{\mu U(x_{n}, t_n r_{n})} \right\|_{L^p (\mu)}
\le
\left\|M^u_{\mu} \right\|_{L^q (\mu) \to L^q  (\mu)} 
\left\| \sum_{n \in I}\frac{w_n \mathbf{1}_{U(x_{n}, r_{n})} }{\mu U(x_{n}, r_{n})} \right\|_{L^p (\mu)}.
\end{equation}
\end{theorem}

\begin{proof} Of course,  if 
$\|M^u_{\mu} \|_{L^q (\mu) \to L^q  (\mu)} = \infty$, 
there is nothing to show, so assume otherwise.  Recalling that $\|f \|_p =  \sup_{\|g\|_q = 1} \int f g,$ we fix $h\ge 0$ with 
$\|h\|_q = 1$. Then
\begin{equation}\label{overlappf}
\int \left( \sum_{n \in I}\frac{w_n \mathbf{1}_{U(x_{n},  t_n r_{n})} }{\mu U(x_{n}, t_n r_{n})}
 \right) h \ d\mu
 =
\sum_{n \in I}\frac{w_n }{\mu U(x_{n}, t_n r_{n})} \int h \mathbf{1}_{U(x_{n},  t_n r_{n})}  \ d\mu
\end{equation}
 \begin{equation}\label{overlappf1}
 \le 
 \sum_{n \in I} w_n \inf_{x\in U(x_{n}, r_{n})} M^u h(x)  
 \le
 \sum_{n \in I}\frac{w_n }{\mu U(x_{n}, r_{n})} \int M^u h \mathbf{1}_{U(x_{n},   r_{n})}  \ d\mu
 \end{equation}
\begin{equation}\label{overlappf2}
=
\int \left( \sum_{n \in I}\frac{w_n \mathbf{1}_{U(x_{n}, r_{n})} }{\mu U(x_{n}, r_{n})} \right) M^u h \ d\mu
\le
\left\| \sum_{n \in I}\frac{w_n \mathbf{1}_{U(x_{n}, r_{n})} }{\mu U(x_{n}, r_{n})} \right\|_p
\left\|M^u_{\mu} h \right\|_q.
\end{equation}
Taking the supremum over all $h\ge 0$ with 
$\|h\|_q = 1$, yields the result.
\end{proof}

\begin{remark} \label{count} Let $(X, d, \mu)$ be a metric measure space and let 
$x\in X$. We use repeatedly the fact that countable unions of countable
collections are countable,  in order to show that if $w_\alpha > 0 $ for uncountably many
indices 
$\alpha$, then
$$
\left\| \sum_{\alpha \in I}\frac{w_\alpha \mathbf{1}_{U(x_{\alpha}, r_{\alpha})} }{\mu U(x_{\alpha}, r_{\alpha})} \right\|_p = \infty.
$$
Note first that there is a  $T > 0$ such that for an uncountable set $I_1$ of  $\alpha$'s, 
$r_\alpha < T$, $w_\alpha > 1/T$, and 
$U(x_{\alpha}, r_{\alpha}) \subset  U (x, T)$. It now follows from Jensen's inequality
that
$$
 \frac{1}{\mu
(U(x, T))^{1/p} }  \left\| \sum_{\alpha \in I}\frac{w_\alpha \mathbf{1}_{U(x_{\alpha}, r_{\alpha})  } }{\mu U(x_{\alpha}, r_{\alpha})} \right\|_p 
\ge
 \left\|  \frac{ \mathbf{1}_{U(x, T)} }{\mu
(U(x, T))^{1/p}  }  \sum_{\alpha \in I_1}\frac{w_\alpha \mathbf{1}_{U(x_{\alpha}, r_{\alpha})} }{\mu U(x_{\alpha}, r_{\alpha})} \right\|_p 
$$
$$
\ge
 \left\|  \frac{ \mathbf{1}_{U(x, T)} }{\mu
(U(x, T))}  \sum_{\alpha \in I_1}\frac{w_\alpha \mathbf{1}_{U(x_{\alpha}, r_{\alpha})} }{\mu U(x_{\alpha}, r_{\alpha})} \right\|_1
= 
\sum_{\alpha \in I_1} w_\alpha 
= \infty.
$$
\end{remark}

\begin{remark} Of course, by monotone convergence one can pass from finite to countable collections of balls,
weights and dilations, so from now on we only consider finite sequences in the statements of results.
\end{remark}

\begin{remark} As indicated in the introduction, there are several cases where the boundedness of the uncentered  maximal operator
is known to hold without the measure being doubling, often using instead strong uniformity assumptions. For instance, it is shown in \cite[Theorem 1a]{Io} that for non-compact symmetric
spaces, $M^u$ is bounded from $L^q$ to $L^q$ in the sharp range of exponents $q\in (2, \infty]$
(not surprisingly, the centered operator $M$ has better boundedness properties, cf. \cite{Str}).
On the other hand, local comparability (a rather mild uniformity condition) of the measure 
on a geometrical doubling metric space does imply the weak type $(1,1)$ of 
$M^u$ (cf. \cite[Theorem 5.1]{Al1}, and \cite{Al2} for additional information on local comparability).
\end{remark}

Next we present a more general version of Theorem \ref{normalizedBoman}, 
by examining the proof and assuming precisely what is needed. In particular, only
the  uncentered maximal operator localized at large radii  is required to be bounded, and the
smaller sets (not necessarily balls) only need to be contained in the corresponding (larger) balls 
(and have positive measure).

\begin{theorem}  \label{generalizedBoman}  Let $(X, d, \mu)$ be a
metric measure space, fix $1 < p < \infty$, and set $q = p/(p-1)$.  Let 
$R > 0$, let $\{w_n\}_{1 \le n \le N}$ be a 
finite sequence of weights, 
let 
$\{E_n \}_{1 \le n \le N}$ be a sequence of sets of positive measure, and let
$\{U(x_n, r_n)\}_{1 \le n \le N}$ be a  sequence of balls 
such that for $1 \le n \le N$, $r_n > R$ and $E_n \subset U(x_n, r_n)$. Then
\begin{equation}\label{overlap}
\left\| \sum_{1 \le n \le N}\frac{w_n \mathbf{1}_{U(x_{n},  r_{n})} }{\mu U(x_{n}, r_{n})} \right\|_{L^p (\mu)}
\le
\left\| M^u_{\{ R < r\}, \mu} \right\|_{L^q (\mu) \to L^q  (\mu)} 
\left\| \sum_{1 \le n \le N}\frac{w_n \mathbf{1}_{ E_n } }{\mu E_n } \right\|_{L^p (\mu)}.
\end{equation}
\end{theorem}

\begin{proof}   Arguing as before, we note that for
$1 \le n \le N$, given $h\ge 0$ with 
$\|h\|_q = 1$, we have
\begin{equation}\label{overlappfgen}
\frac{w_n }{\mu U(x_{n}, r_{n})} \int h \mathbf{1}_{U(x_{n},  r_{n})}  \ d\mu
\le 
 w_n \inf_{x\in U(x_{n}, r_{n})} M^u_{\{ R < r\}, \mu} h(x)  
  \end{equation}
\begin{equation}\label{overlappfgen2}
\le
 w_n \inf_{x\in E_n  } M^u_{\{ R < r\}, \mu} h(x)  
 \le  \frac{w_n }{\mu E_n  } \int M^u_{\{ R < r\}, \mu} h \mathbf{1}_{E_n  }  \ d\mu.
 \end{equation}
 \end{proof}
 
 \begin{remark} Note that if we replace $\left\| M^u_{\{ R < r\}, \mu} \right\|_{L^q (\mu) \to L^q  (\mu)} 
$ in (\ref{overlap})  
  by $\left\| M^u_{\mu} \right\|_{L^q (\mu) \to L^q  (\mu)} 
$, then the reference to $R > 0$ in the preceding theorem can be eliminated.
\end{remark}

Under stronger hypotheses,
and up to some constants,  it is possible to reverse the inequality in Boman's covering lemma. Note that the doubling condition
entails not only that the usual maximal functions are bounded, but boundedness also holds
for the larger operators 
obtained by integrating over expanded balls.

\begin{definition}\label{maxfun} Let $(X, d, \mu)$ be a metric measure space and let $g$ be  a locally integrable function 
on $X$. Given $t \ge 1$ we set, for each  $x\in X$, \begin{equation}\label{HLMFct}
M^t_{\mu} g(x) := \sup _{\{r > 0  :  \ 0 < \mu (B^{cl}(x, r))\}}  \frac{1}{\mu
(B^{cl}(x, r))} \int _{B^{cl}(x, t r)}  |g| d\mu.
\end{equation}
The corresponding  uncentered version  is defined via
\begin{equation}\label{HLMFu}
M^{u t}_{\mu} g(x) := \sup _{\{r > 0, y\in X :  \  d(x,y) < t r \mbox{ and } \mu B^{cl}(y, r)  > 0\}} 
\frac{1}{\mu
(B^{cl}(y, r))} \int _{B^{cl}(y, t r)}  |g (u)| d\mu (u).
\end{equation}
\end{definition}

It is well known that if $\mu$ is doubling with constant $C$, then $M^u$ is of weak type
$(1,1)$ with constant $C^3$ (by a $5B$-covering argument, where the radii of balls
is expanded by 5). Likewise,  $M$ is of weak type
$(1,1)$ with constant $C^2$, by a $3B$-covering argument. By interpolation, it follows that 
if $1 < p < \infty$, then
 $$
\|M^{u }_{\mu}\|_{L^p\to L^p} 
 \le  \frac{C^{3/p}p }{p-1}
\ \ \  \text{ and  } \ \ \ 
 \|M_{\mu}\|_{L^p\to L^p} 
 \le  \frac{C^{2/p}p }{p-1}
 $$ 
(cf. \cite[p. 46, Exercise 1.3.3 (a)]{Gra}).

\begin{lemma}  \label{largemax}  Let $(X, d, \mu)$ be a
metric measure space, where $\mu$ is doubling with constant $C$. Fix $t >1$ and let $k\ge 0$ be the
unique 
integer that satisfies $2^{k - 1} < t \le 2^k$. Then  $M^{u t}_{\mu}$ is of
weak type $(1,1)$ with constant $C^{3 + k}$, and  $M^{ t}_{\mu}$ is of
weak type $(1,1)$ with constant $C^{2 + k}$. Hence 
$
\|M^{u t}_{\mu}\|_{L^p\to L^p} 
 \le  \frac{p }{p-1} C^{(3 + k) /p},
 $
and 
$ \|M^t_{\mu}\|_{L^p\to L^p} 
 \le  \frac{p }{p-1} C^{(2 + k)/p}.
 $ 
 \end{lemma}

\begin{proof} This is immediate from the pointwise inequalities
$M^{u t}_{\mu}  g (x) \le
C^k  M^{u}_{\mu}  g (x)$ and $M^{ t}_{\mu}  g (x) \le
C^k  M^{u}_{\mu}  g (x)$, for every locally integrable $g$.
\end{proof} 

Since for any $t \ge 1$,
\begin{equation*}
\frac{1}{\mu U(x, r)} \int h \mathbf{1}_{U(x,  r)}  \ d\mu
\le 
  \inf_{x\in U(x , t r)} M^{u t }_{\mu} h(x),  
  \end{equation*}
with the same proof as before, we get the following

\begin{theorem}  \label{rBoman} {\bf (Reverse Boman)} Let $(X, d, \mu)$ be a
metric measure space, with $\mu$ doubling,  fix $1 < p < \infty$,  and set $q = p/(p-1)$.  Let $\{w_n\}_{1 \le n \le N}$ be a 
finite   sequence of weights,  let
$\{t_n\}_{1 \le n \le N}$ be a 
  sequence of dilations, with  $1 \le t_n \le t$,
and let $\{U(x_n, r_n)\}_{1 \le n \le N}$ be a sequence of balls. Then
\begin{equation}\label{roverlap1}
\left\| \sum_{1 \le n \le N}\frac{w_n \mathbf{1}_{U(x_{n},   r_{n})} }{\mu U(x_{n}, r_{n})} \right\|_p
\le
\left\|M^{ u t }_{\mu} \right\|_{L^q (\mu) \to L^q  (\mu)} 
\left\| \sum_{1 \le n \le N}\frac{w_n \mathbf{1}_{U(x_{n},  t_n r_{n})} }{\mu U (x_{n}, t_n r_{n})} \right\|_p.
\end{equation}
\end{theorem}

\begin{example} Note that the bounds given in (\ref{roverlap1}) must necessarily grow
with $t$, as can be seen even from the simplest example: 
Consider Lebesgue measure $\lambda$ on $\mathbb{R}$; let $w_i = 1$, and 
$t_i = t = 2^k$. Then
$$
\left\|\frac{\mathbf{1}_{B(0, 1)} }{\lambda B(0, 1)} \right\|_p
=
2^{-1/q} 
= 
2^{k/q}  \left\|\frac{\mathbf{1}_{B(0, 2^k)} }{\lambda B(0, 2^k)} \right\|_p.$$
 \end{example}

Recall that both the centered and the uncentered
maximal functions are lower semicontinuous; in the case of the uncentered maximal function
this is obvious; in the centered case it follows from the fact that if
$c < \frac{ 1}{\mu B^{cl}(x, r)} \int_{B^{cl}(x, r)} |f| d\mu  $, then there exists a $t > 1$
such that   $c < \frac{ 1}{\mu B^{cl}(x, t r)} \int_{B^{cl}(x, r)} |f| d\mu  $. 

\vskip .2 cm

Next we
relax  the boundedness  hypothesis on the uncentered operator $M^u$ 
 to the boundedness of the centered operator 
$M$. In exchange, the result is  no longer obtained for arbitrary
contractions $t < 1$ on the right hand side: they must be ``small enough". 

\begin{theorem}  \label{centeredBoman}  Let $(X, d, \mu)$ be a
metric measure space, fix $1 < p < \infty$, and set $q = p/(p-1)$.  Let $\{B^{cl}(x_n, r_n)\}_{1 \le n \le N}$ be a sequence of closed 
balls with centers in the support of
$\mu$. If 
$\|M_{\mu} \|_{L^q (\mu) \to L^q  (\mu)} < \infty$, then
there exists a numerical sequence 
$\{T_n\}_{1 \le n \le N}$, with  $0 < T_n < 1$ depending  on $\{B^{cl}(x_n, r_n)\}_{1 \le n \le N}$, such that whenever $0 < t_n \le T_n$, 
for every sequence of weights
 $\{w_n\}_{1 \le n \le N}$, we have  

\begin{equation}\label{coverlap}
\left\| \sum_{1 \le n \le N}\frac{w_n \mathbf{1}_{B^{cl} (x_{n}, r_{n})} }{\mu B^{cl}(x_{n}, r_{n})} \right\|_p
\le
2
\left\|M_{\mu} \right\|_{L^q (\mu) \to L^q  (\mu)} 
\left\| \sum_{1 \le n \le N}\frac{w_n \mathbf{1}_{B^{cl}(x_{n},   t_n  r_{n})} }{\mu B^{cl} (x_{n},  t_n  r_{n})} 
\right\|_p.
\end{equation} 
\end{theorem}

\begin{proof} Fix $h\ge 0$ with 
$\|h\|_q = 1$, select $0 < \delta_n < r_n$ so that 
$\mu B^{cl} (x_{n},  \delta_n +  r_{n}) < 2 \mu B^{cl} (x_{n},   r_{n})$ 
 (here we use the fact that balls are assumed to be closed) 
and set $T_n := \delta_n/r_n$.
 Then, for all $0 < t_n \le T_n$, and all $y\in B^{cl} (x_{n},   t_{n} r_n) \subset B^{cl} (x_{n},   \delta_{n})$,
$$
\frac{1}{\mu B^{cl} (x_{n},  r_{n})} \int h \mathbf{1}_{B^{cl} (x_{n},  r_{n})}  \ d\mu
<
 \frac{2 }{\mu B^{cl} (x_{n}, (1 + t_n) r_{n})} \int h \mathbf{1}_{B^{cl} (x_{n},  r_{n})}  \ d\mu
$$
$$
\le
 \frac{2 }{\mu B^{cl} (y,  (1 + d(y, x_n) ) r_{n})} \int h \mathbf{1}_{B^{cl} (y,  (1 + d(y, x_n) ) r_{n})}  \ d\mu
 \le
2 M h(y).
$$
Therefore
\begin{equation}\label{overlappf}
\int \left( \sum_{1 \le n \le N}\frac{w_n \mathbf{1}_{B^{cl} (x_{n},  r_{n})} }{\mu B^{cl} (x_{n},  r_{n})}
 \right) h \ d\mu
 =
\sum_{1 \le n \le N}\frac{w_n }{\mu B^{cl} (x_{n}, r_{n})} \int h \mathbf{1}_{B^{cl} (x_{n},   r_{n})}  \ d\mu
\end{equation}
 \begin{equation}\label{overlappf1}
 \le 
2  \sum_{1 \le n \le N} w_n \inf_{x\in B^{cl} (x_{n}, t_n r_{n})} M h(x)  
 \le
2  \sum_{1 \le n \le N}\frac{w_n }{\mu B^{cl} (x_{n}, t_n r_{n})} \int (M h) 
 \mathbf{1}_{B^{cl} (x_{n}, t_n  r_{n})}  \ d\mu
 \end{equation}
\begin{equation}\label{overlappf2}
=
2 \int \left( \sum_{1 \le n \le N}
\frac{w_n \mathbf{1}_{B^{cl} (x_{n}, t_n  r_{n})} }{\mu B^{cl} (x_{n}, t_n  r_{n})} \right) M h \ d\mu
\le
2 \left\| \sum_{1 \le n \le N}\frac{w_n \mathbf{1}_{B^{cl} (x_{n}, t_n r_{n})} }{\mu B^{cl} (x_{n}, t_n  r_{n})} \right\|_p
\left\|M_{\mu} h \right\|_q.
\end{equation}
\end{proof}

\begin{remark} Regarding Theorem \ref{centeredBoman} in the special 
case of Lebesgue measure $\lambda^d$ in $\mathbb{R}^d$,  with different distances (and hence different  balls)  defined by norms,
 it is
known for some balls
if $p > 1$, and for all balls  
  if $p > 3/2$, that $\left\|M_{\mu} \right\|_{L^p (\mu) \to L^p  (\mu)}$ can be
given bounds  that are independent of the dimension, see \cite{DeGuMa} and the references contained therein.

 For instance, if $p=2$, by \cite[Theorem 5.2]{DeGuMa}, we have  $\left\|M_{\mu} \right\|_{L^2 (\mu) \to L^2  (\mu)} < 140$, where the constant $140$ is 
independent of $d$, and of the norm considered. Thus, (\ref{coverlap}) becomes
\begin{equation}\label{coverlapLeb}
\left\| \sum_{1 \le n \le N}\frac{w_n \mathbf{1}_{B^{cl} (x_{n}, r_{n})} }{\lambda^d B^{cl}(x_{n}, r_{n})} \right\|_2
\le
280 
\left\| \sum_{1 \le n \le N}\frac{w_n \mathbf{1}_{B^{cl}(x_{n},   t_n  r_{n})} }{\lambda^d B^{cl} (x_{n},  t_n  r_{n})} 
\right\|_2, 
\end{equation}
and if $p \in (1,2)$, then 280 can be replaced by $2 (140)^{2/q}$ by a simple application of Jensen's inequality (cf. 
\cite[Theorem 2.10]{Al1} for more details). Note however that the $T_n$'s must be very
small, depending on the dimension $d$. For instance, for large values of $d$ and every
$n$,  we can take $T_n = 1/(2 d)$, since $(1 + 1/(2 d))^d \approx e^{1/2} < 2$.

Note also that the constant $2 = 1+1$ in (\ref{coverlap})  can
be replaced by $1 + \varepsilon$, for any fixed $\varepsilon > 0$, simply by choosing smaller
values of $T_n$.
\end{remark}

\section{Applications to maximal operators}

\begin{definition}\label{Bomanprop} Let $1 \le p \le \infty$. We say that the   metric measure 
space $(X, d, \mu)$ 
has the {\em  $p$-Boman covering  property} (or $p$-Boman  overlapping property) with constant $C > 0$,  if for all
finite sequences of closed
balls   $\{B^{cl}(x_n, r_n)\}_{1 \le n \le N}$ with $x_n$ in the support of $\mu$,   of weights $\{w_n\}_{1 \le n \le N}$,  $w_n >0$,  
and of dilations  $\{t_n\}_{1 \le n \le N}$, $t_n \ge 1$, 
\begin{equation}\label{overlapC}
\left\| \sum_{1 \le n \le  N}\frac{w_n \mathbf{1}_{B^{cl}(x_{n},  t_n r_{n})} }{\mu B^{cl}(x_{n}, t_n r_{n})} \right\|_p
\le
C
\left\| \sum_{1 \le n \le  N}\frac{w_n \mathbf{1}_{B^{cl}(x_{n}, r_{n})} }{\mu B^{cl}(x_{n}, r_{n})} \right\|_p.
\end{equation}
The    generalized $p$-Boman covering property  with constant $C > 0$,  
is defined analogously, but with measurable sets of positive measure replacing the smaller balls.
More precisely,  {\em the generalized $p$-Boman  covering property} holds
if for all
finite sequences of closed
balls   $\{B^{cl}(x_n, r_n)\}_{1 \le n \le N}$ with $x_n$ in the support of $\mu$,   of weights $\{w_n\}_{1 \le n \le N}$,  $w_n >0$,  
and of   sets $\{E_n\}_{1 \le n \le N}$ with  $\mu E_n > 0$ and 
$E_n \subset  B^{cl}(x_{n},   r_{n})$,
we have  
\begin{equation}\label{goverlapC}
\left\| \sum_{1 \le n \le  N}\frac{w_n \mathbf{1}_{B^{cl}(x_{n},   r_{n})} }{\mu B^{cl}(x_{n}, r_{n})} \right\|_p
\le
C
\left\| \sum_{1 \le n \le  N}\frac{w_n \mathbf{1}_{ E_n } }{\mu E_n } \right\|_p.
\end{equation}
\end{definition}

\begin{remark} 
Observe that the generalized  $1$-Boman overlapping property always holds, with $C = 1$, since
\begin{equation}
\left\| \sum_{1 \le n \le  N}\frac{w_n \mathbf{1}_{B^{cl}(x_{n},  r_{n})} }{\mu B^{cl}(x_{n},  r_{n})} \right\|_1
=
\sum_{1 \le n \le  N} w_n
=
\left\| \sum_{1 \le n \le N}\frac{w_n \mathbf{1}_{ E_n } }{\mu E_n } \right\|_1.
\end{equation}

On the other hand, the  $\infty$-Boman overlapping property will rarely hold (it holds if $X$ is finite, for
instance). It fails for very well behaved spaces, such as the Lebesgue probability $\lambda$ in $[0,1]$, as we show next.
For $1 \le n \le N$, set $ B^{cl}(x_{n}, r_{n}) = [2^{-n}, 2^{-n + 1} ]$, $w_n = 2^{-n} = \lambda B^{cl}(x_{n}, r_{n}) $, and $t _n = 3$.
Since $x_n = 3 \times 2^{-n - 1}$ and $r_n = 2^{-n - 1}$, 
$$
\left\| \sum_{1 \le n \le  N}\frac{w_n \mathbf{1}_{B^{cl}(x_{n},  t_n r_{n})} }{\lambda B^{cl}(x_{n}, t_n r_{n})} \right\|_\infty
= \frac{N}{3},
$$
while
$$
\left\| \sum_{1 \le n \le  N}\frac{w_n \mathbf{1}_{B^{cl}(x_{n}, r_{n})} }{\lambda B^{cl}(x_{n}, r_{n})} \right\|_\infty
= 1.
$$
Thus,  the $p$-Boman overlapping property and its generalized versions are interesting mostly for $1 < p < \infty$.
\end{remark}

 In \cite[p. 73]{HKST} the conclusion of the Vitali covering theorem is turned into a definition as follows:

\begin{definition}\label{vitali} 
We say that   $(X, d, \mu)$    is a {\em Vitali} metric measure space 
 if  for all  $A \subset  X$  and 
every covering $\mathcal{ B}$ of $A$ 
by closed balls satisfying
$$
\sup_{x\in A} \inf\{r > 0: B^{cl} (x,r) \in  \mathcal{ B} \} = 0,
$$
there
exists a  disjoint subcollection  $\mathcal{ C}\subset \mathcal{ B}$  such that
$$
\mu (A\setminus \cup  \mathcal{ C})  = 0.
$$
\end{definition}
 
Of course, if $\mu$ is doubling then $X$ is Vitali, but the latter is a much weaker hypothesis. For instance, it holds 
for volume on a Riemannian manifold (cf. \cite[Examples p. 76-77]{HKST} for this 
and for other examples).

A standard application of H\"older's inequality leads us from the generalized $q$-Boman covering property to the
weak type $(p,p)$ of the centered maximal operator $M$, 
and from the $q$-Boman covering property to the
weak type $(p,p)$ of $M$
for Vitali metric measure spaces. 
As we shall see below,
the $q$-Boman covering property is strictly weaker than the strong $(p,p)$ type of $M^u$. 
Thus, proving the $q$-Boman covering property in cases where the boundedness of $M^u$ is not
known, or it is known to fail, may yield new results for the centered operator.

 But while this property
ought to be easier to obtain than the strong type of the maximal operator, since it is weaker, in fact
it is not obvious  how to do so. Hence, we define weaker properties, in the line of the
conclusion of Theorem  \ref{centeredBoman}, which still yield the above results.

\begin{definition}\label{weakBomanprop} Let $1 \le p \le \infty$. We say that the   metric measure space $(X, d, \mu)$ 
has the {\em  weak $p$-Boman covering property}   with constant $C > 0$,  if for every
$r>0$ there exists a $T = T(r) \in (0,1)$ such that for all 
finite sequences of closed
balls   $\{B^{cl}(x_n, r_n)\}_{1 \le n \le N}$ with $x_n$ in the support of $\mu$,   of weights $\{w_n\}_{1 \le n \le N}$,  $w_n >0$,  
and of contractions   $\{t_n\}_{1 \le n \le N}$, $0 < t_n \le T(r_n)$, we have
\begin{equation}\label{woverlapC}
\left\| \sum_{1 \le n \le  N}\frac{w_n \mathbf{1}_{B^{cl}(x_{n},   r_{n})} }{\mu B^{cl}(x_{n}, r_{n})} \right\|_p
\le
C
\left\| \sum_{1 \le n \le  N}\frac{w_n \mathbf{1}_{B^{cl}(x_{n}, t_n  r_{n})} }{\mu B^{cl}(x_{n}, t_n  r_{n})} \right\|_p.
\end{equation}
The generalized weak $p$-Boman  covering property is defined analogously, with
sets of positive measure contained in the smaller balls replacing the latter.
\end{definition}

\begin{theorem}\label{M}  Let $1 < p < \infty$ and let 
 $(X, d, \mu)$ be a metric measure space satisfying a generalized weak $q$-Boman covering property with constant $C$. 
Then the centered maximal operator
$M$ is of weak type $(p,p)$, with $\|M\|_{L^p\to L^{p, \infty}} \le C$.
 \end{theorem}

 \begin{proof} Fix $\varepsilon > 0$, let $a > 0$ be such that $\mu \{M f > a\} > 0$, and let $f\in L^p (\mu)$. 
Now let $E :=  \{M f > a\} \cap 
\operatorname{supp}\mu$, where  $\operatorname{supp}\mu$ denotes 
 the support of $\mu$. Since
the complement of
the support $(\operatorname{supp}\mu)^c := \cup \{ B(x, r): x \in X, \mu B(x,r) = 0\}$
has measure zero by $\tau$-additivity, $\mu E =  \mu \{M f > a\}$.

 For each $x\in E$, select $B^{cl}(x,r)$  such that $a \mu  B^{cl}(x,r) < \int_{B^{cl}(x,r)} |f| \ d\mu$, 
and choose  $T_x\in (0, T(r)]$ with $B^{cl}(x, T_x r) \subset \{M f > a\}$. Define
$\mathcal{ B}:= \{B(x,t_x r) : x\in E, 0 < t_x < T_x \} $. By $\tau$-additivity, there is a finite
subcollection $\mathcal{ C} = \{ B (x_i ,t_i r_i): 1 \le  i \le N\}$ such that 
$$
(1 - \varepsilon) \mu E 
= (1 - \varepsilon) \mu \cup  \mathcal{ B} 
< 
\mu \cup \mathcal{ C} 
\le
\mu \cup_{i= 1}^N B^{cl} (x_i ,t_i r_i).
$$
Furthermore, we may assume that for each $1 \le i \le N$, 
$\mu B^{cl} (x_i ,t_i r_i) \setminus \cup_{j\ne  i}^N B^{cl} (x_j ,t_j r_j) > 0$,
for otherwise we simply disregard whichever ball is almost contained in the union of
the others, and this does not change the measure of the finite  union. 
Define the disjointifications $E_1 = B^{cl} (x_1 ,t_1 r_1)$, and for $j > 1$, 
$E_j =  B^{cl} (x_j ,t_j r_j) \setminus \cup_{i= 1}^{j - 1} B^{cl} (x_i ,t_i r_i)$. 
It follows from the generalized weak $q$-Boman covering property, with 
$w_i := \mu E_i >0$, that
$$
(1 - \varepsilon) \mu  \{M f > a\}
\le
\mu \cup_{i= 1}^N B^{cl} (x_i ,t_i r_i) 
=
 \sum_{i= 1}^N \mu E_i
$$
$$
= \sum_{i= 1}^N 
 \frac{ \mu E_i }{ \mu B^{cl} (x_i ,  r_i)}\int_X  \mathbf{1}_{  B^{cl} (x_i ,  r_i)}  d \mu
 \le
 \sum_{i=1}^N  \frac{ \mu E_i }{ \mu B^{cl} (x_i ,  r_i)} 
  \frac{1}{a }\int_X  |f| \mathbf{1}_{  B^{cl} (x_i ,  r_i)}  d \mu
 $$
$$
=
 \frac{1}{a }\int_X  |f| \sum_{i=1}^N  \frac{ \mu E_i }{ \mu B^{cl} (x_i ,  r_i)} \mathbf{1}_{  B^{cl} (x_i ,  r_i)}  d \mu
\le
 \frac{1}{a }\left\|f\right\|_p \left\| \sum_{i=1}^N  \frac{ \mu E_i }{ \mu B^{cl} (x_i ,  r_i)} \mathbf{1}_{  B^{cl} (x_i ,  r_i)}  
\right\|_q
$$
$$
\le
\frac{C}{a }\left\|f\right\|_p \left\| \sum_{i=1}^N \mathbf{1}_{  E_i }  
\right\|_q
\le 
\frac{C}{a }\left\|f\right\|_p \mu \{M f > a\}^{1/q}.
$$
Dividing both sides of 
$$
(1 - \varepsilon)  \mu \{M f > a\}
\le 
\frac{C}{a }\left\|f\right\|_p \mu \{M f > a\}^{1/q}.
$$  by  $\mu \{M f > a\}^{1/q}$ and letting $\varepsilon \downarrow 0$ we obtain the result. 
 \end{proof}
 
Next we indicate the changes needed in the proof when only the weak $q$-Boman covering
property is assumed, but in exchange $X$ is taken to be Vitali.

\begin{theorem}\label{VM}  Let $1 < p < \infty$ and let 
 $(X, d, \mu)$ be a Vitali  metric measure space satisfying a weak $q$-Boman covering property with constant $C$. 
Then the centered maximal operator
$M$ is of weak type $(p,p)$, with $\|M\|_{L^p\to L^{p, \infty}} \le C$.
 \end{theorem}

 \begin{proof} With the same notation as in the previous proof, for 
 each $x\in E$ we select $B^{cl}(x,r)$  such that $a \mu  B^{cl}(x,r) < \int_{B^{cl}(x,r)} |f| \ d\mu$, 
and choose  $T_x\in (0, T(r)]$ with $B^{cl}(x, T_x r) \subset \{M f > a\}$. Define
$\mathcal{ B}:= \{B^{cl}(x,t_x r) : x\in E, 0 < t_x < T_x \} $, so $\mathcal{ B}$ is a Vitali cover of $E$.  Let 
$\mathcal{ C}\subset \mathcal{ B}$ be a (finite or countable)  disjoint subcollection with
$\mu (E\setminus \cup  \mathcal{ C})  = 0$, say, 
$\mathcal{ C} = \{ B^{cl} (x_i ,t_i r_i):  i \in I\}$, where either  $I = \{1, \dots , n\}$ for some 
$n \ge 1$, or
$I = \mathbb{N}\setminus \{0\}$. By countable additivity there is a finite  $N\ge 1$ such that
$$
(1 - \varepsilon) \mu E 
< 
\mu \cup_{i= 1}^N B^{cl} (x_i ,t_i r_i) = \sum_{i= 1}^N \mu B^{cl} (x_i ,t_i r_i) ,
$$
and now, taking  $w_i := \mu B^{cl} (x_i ,t_i r_i) >0$, the rest of the argument is as before.  \end{proof}

\begin{remark} We do not know if the Vitali assumption can be removed from the preceding
theorem. It can be relaxed, however, to the assumption that the Lebesgue differentiation
theorem holds for the functions in $L^p$, by \cite{Ha}.
\end{remark}

After some preparation, we show that there exists a metric measure space where the uncentered maximal operator
satisfies no strong bounds for $1 < q < \infty$, and still $p$-Boman holds for every $p\in [1, \infty)$.

\begin{definition} A Borel measure $\mu$ on $(X,d)$ is {\em doubling at large scales}, or  doubling localized at radii larger
than $R$,   if there exists an $R > 0$ and a 
$C \ge 1 $ such that for all $r>R $ and all $x\in X$, $\mu (B(x, 2 r)) \le C\mu(B(x,r)) < \infty$. 
\end{definition}

\begin{definition} A Borel measure $\mu$ on $(X,d)$ is {\em doubling on a subset} $E \subset X$,
  if there exists  a 
$C \ge 1 $ such that for all  $x\in E$, $\mu (B(x, 2 r)) \le C\mu(B(x,r)) < \infty$. 
\end{definition}

The next example will be used later, in Theorem \ref{overlapnotdoubling}.

\begin{example} \label{locdoub} {\em There exists a   Vitali  metric measure space 
$(X, d, \mu)$ such that   for each $1 \le p < \infty$,  $M^u$ is not of weak type $(p,p)$, 
 but $\mu$ it is doubling localized at radii larger than $R > 0$, for
any such $R$.  In fact, if $R\ge 2^{-n}$, we can take $C = 2^{3 + n} $.}

\vskip .2 cm

The set $X$ will be a suitable subset of $\mathbb{R}^2$ with distance defined by the
$\ell_1$ norm: Given $x = (x_1, x_2)$ and $y = (y_1, y_2)$, $d(x,y) = \|x -y\|_1 = |x_1 - y_1| + |x_2 - y_2|$. By the Besicovitch covering theorem, $X$ is Vitali. 

\vskip .2 cm

Let $A_n := [4 n, 4 n + 2^{-n}] \times \{2^{-n}\}$, and let  $A := (-\infty, \infty) \times \{0\}$ be  the $x$ axis.
Set $X := A \cup (\cup_{n \ge 0} A_n)$, and define $\mu$ restricted to $A$ as linear Lebesgue measure $\lambda$,
while $\mu$ restricted to $A_n$ is $ 2^{-n} \lambda$, the 
 linear Lebesgue measure scaled by a factor of $ 2^{-n}$; thus,  
$\mu A_n =  4^{-n}$. Clearly the failure of doubling occurs at small radii: 
$\mu B_1^{cl}((4 n,  2^{-n}),  2^{-n}) =  4^{-n}$, while
 $\mu B_1^{cl}((4 n,  2^{-n}),  2^{-n + 1}) =  4^{-n} + 2^{1 -n}$. 
 
 \vskip .2 cm

Lack of  doubling  also follows from the fact  that $M^u_\mu$ is not of strong type $(p,p)$ for
any $1 < p < \infty$ (and hence, it is not of weak type $(p,p)$ either). Since $\|\mathbf{1}_{A_n} \|_p^p =  4^{-n}$,
and $M^u_\mu \mathbf{1}_{A_n}  \equiv 1 $ on $[4 n, 4 n + 2^{-n}] \times \{0\}$, we have
    $\|M^u_\mu \mathbf{1}_{A_n}\|^p_p > 2^{n} \|\mathbf{1}_{A_n} \|_p^p $.
    
\vskip .2 cm

We prove by induction on $n$, where $R \ge 2^{-n}$, that 
$\mu$ is doubling in the large. 
For every $x\in A$, 
$$
\mu B(x, 2r) = \mu (A \cap B(x, 2r)) +  \mu (\cup_{n \ge 0} A_n \cap B(x, 2r))
$$
$$
\le 
 2 \mu (A \cap B(x, 2r)) 
= 4   \mu (A \cap B(x, r)) \le  4 \mu (B(x, r)),
$$
so on $A$, $\mu$ is actually doubling, and we only need to consider balls with centers in  
$\cup_{n \ge 0} A_n$;  this is always assumed from now on. 
 Denote by $\pi_1 (a,b) := (a,0)$ the projection onto the first coordinate, and observe that for every
$x \in \cup_{n \ge 0} A_n$, 
\begin{equation}\label{double}
\mu B(x, 2r) \le \mu ( B(\pi_1 x, 2r)) \le 8r.
\end{equation}
Now the claim to be shown  is that if $R \ge  2^{-n}$,  then $C = C(R) \ge \mu B(x, 2r) / \mu ( B(x, r)) $ can be taken to be equal to
 $ 2^{3 + n}$.
Note  that if  $ R \ge  2^{-n}$, then there is no loss of generality in assuming that $R =   2^{-n}$,
since more radii are being considered in this case.

Let $n= 0$, so $r > R=1$. 
If $x\in A_N$, $N > 0$, then
$$
\mu B(x, r) 
\ge
 \mu (A \cap B(x, r)) 
\ge 2 r - 1 \ge r;
$$
hence,  we can take $C= 2^{3 + 0}$.

And if $x\in A_0$,  for $r \ge 2$,
$$
\mu B(x, r) 
\ge
 \mu (A \cap B(x, r)) 
\ge 2 r - 2 \ge r,
$$
while if  $1 < r < 2$,  then
$$
\mu B(x, 2r) \le 1 + 4 r - 2 < 4r,
$$
and
$
\mu B(x, r) 
\ge 1 > r/2,
$
so once again we can take $C= 2^{3 + 0}$.

The general case $ R \ge  2^{-n}$ is obtained in a similar fashion. If $r >  2^{-n + 1 }$ we take  $ R =  2^{-n + 1}$
and apply the induction hypothesis, which we assume for $n - 1 \ge 0$,  and prove for $n$. 
As before, we may suppose that   $ R =  2^{-n}$ and
 $2^{-n + 1 } \ge  r >  2^{-n}$. Let  $x\in A_N$. 
The restriction of $\mu$ to $A_N$ is doubling with constant 2, so 
 if $N < n - 1$, then $
\mu B(x, 2r) \le 2 \mu (B(x, r)).$  
If $N = n-1$,  then
$
\mu B(x, r) 
\ge
 2^{- n + 1} r,
$
if 
 $N = n$,  then
$
\mu B(x, r) 
\ge
 2^{- n } r,
$
and if $N > n$,  then
$$
\mu B(x, r) 
\ge
 \mu (A \cap B(x, r)) 
\ge 2 r - 2^{-n} >  r,
$$
so it follows from (\ref{double}) that in every case  $ \mu B(x, 2r) / \mu ( B(x, r))  \le 2^{3 + n}$.
 \end{example}

The standard argument from doubling to weak type $(1,1)$ yields 
the corresponding results when we localize at large scales, and
when we work with measures doubling on subsets.

\begin{theorem}\label{M^uloc}  Let 
 $(X, d, \mu)$ be a  metric measure space.   If $\mu$ is doubling at large scales, say,  localized at radii larger
than $R > 0$, with doubling constant $C = C(R)$,
then the uncentered localized maximal operator
 $M^u_{\{ R < r\}, \mu} $  is of weak type $(1,1)$ with constant $C^3$, and thus, of strong type $(p,p)$ for all $p > 1$, with constant $C =   \frac{C^{3/p}p }{p-1}.$

 Likewise, if $\mu$ is doubling on $E\subset X$, with constant  $C$, then for every $t > 0$ and every
 $g \in L^1 (\mu)$, 
 $$
 \mu (E \cap \{M^u_\mu g >  t\}) \le \frac{C^3  \|g\|_{L^1(\mu)}}{t}.
$$
Thus, if  $f \in L^p (\mu)$ for some $p > 1$, 
$$
\| \mathbf{1}_E M^u_\mu f \|_p \le \frac{C^{3/p}p }{p-1} \|f\|_p.
$$
 \end{theorem}
 
In  metric spaces neither the centers nor the radii of balls  are in general unique. We use the standard convention whereby
when $B$ is used as an abbreviation of $B(x,r)$, we are implicitly  keeping the same $x$ and the same $r$, and when
 $t > 0$, $t B$ is understood as $B(x, tr)$. 
 
 \begin{theorem}\label{overlapnotdoubling}  There exists a Vitali
metric measure space  $(X, d, \mu)$  satisfying a $p$-Boman covering property
for each $1 < p < \infty$, such that for all $q \in (1,  \infty)$, 
 $\left\|M^u_{\mu} \right\|_{L^q (\mu) \to L^q  (\mu)} = \infty$.
  \end{theorem}

\begin{proof} We use the Vitali metric measure space $(X, d, \mu)$  defined in example  \ref{locdoub}.  It was already noted there that 
 $M^u_\mu$ is not of strong type $(q,q)$ for
any $1 < q < \infty$.  Fix $p\in (1,\infty)$.

\vskip .2 cm

To check that for some $C > 0$ (whose value will be determined later)
\begin{equation}\label{first}
\left\| \sum_{1 \le n \le  N}\frac{w_n \mathbf{1}_{B^{cl}(x_{n},  t_n r_{n})} }{\mu B^{cl}(x_{n}, t_n r_{n})} \right\|_p
\le
C
\left\| \sum_{1 \le n \le  N}\frac{w_n \mathbf{1}_{B^{cl}(x_{n}, r_{n})} }{\mu B^{cl}(x_{n}, r_{n})} \right\|_p,
\end{equation}
we split the sum in the left hand side into smaller sums, and then we use the
triangle inequality.
Without loss of generality, we may assume that the balls 
$\{B^{cl}(x_{n}, t_n r_{n}): 1 \le n \le  N\}$ are ordered by decreasing radii $t_n r_n$.
Let $k_1$ be such that for $1 \le n \le k_1$,  $t_n r_n > 1$,  and for  $k_1 <  n \le  N$,
$t_n r_n  \le 1$.
 Note that if $k_1 < 1$ (that is, $t_n r_n \le 1$ always) then the first sum is empty and we 
 can go directly  to the
 second step, while if $k_1 = N$ (so  $t_n r_n > 1$ for every $n$) then the first sum is the whole sum
 and the proof finishes in one step.  So suppose $1 \le  k_1 <  N$. Now by Example 
  \ref{locdoub} 
and Theorem \ref{M^uloc},  the uncentered localized maximal operator
 $M^u_{\{ 2^0 < r\}, \mu} $  is of strong type $(q,q)$, with constant $C =   \frac{8^{3/q}q }{q-1}.$
 Thus
 \begin{equation}\label{ein}
I:= \left\| \sum_{1 \le n \le  k_1}\frac{w_n \mathbf{1}_{B^{cl}(x_{n},  t_n r_{n})} }{\mu B^{cl}(x_{n}, t_n r_{n})} \right\|_p
\le
 C_1
\left\| \sum_{1 \le n \le  N}\frac{w_n \mathbf{1}_{B^{cl}(x_{n}, r_{n})} }{\mu B^{cl}(x_{n}, r_{n})} \right\|_p,
\end{equation}
with $C_1 =  \frac{8^{3/q}q }{q-1}$.

Recall that  $X = A \cup (\cup_{n \ge 0} A_n)$, where 
$A_n := [4 n, 4 n + 2^{-n}] \times \{2^{-n}\}$ and $A := (-\infty, \infty) \times \{0\}$,
and the distance on $X$ is defined by  restriction of the $\ell_1$ norm on the plane.
By reordering the balls 
$\{B^{cl}(x_{n}, t_n r_{n}): k_1 < n \le  N\}$ in such a way that the ones centered at
$A$ appear first in the list, we may suppose that there exists a $k_2 \in (k_1, N]$ such
that if $k_1 < n \le k_2$, then $x_n \in A$, while if $n > k_2$, then $x\in  \cup_{n \ge 0} A_n$.
Of course, if $k_2 = N$ then the set of indices $n$ satisfying $n > k_2$ is empty, and the
proof finishes with the next step. Set
\begin{equation}\label{zwei}
II:= \left\| \sum_{ k_1 < n \le k_2}\frac{w_n \mathbf{1}_{B^{cl}(x_{n},  t_n r_{n})} }{\mu B^{cl}(x_{n}, t_n r_{n})} \right\|_p.
\end{equation}
Note that for every $y \in X$,
\begin{equation*}
\sum_{ k_1 < n \le k_2}\frac{w_n \mathbf{1}_{B^{cl}(x_{n},  t_n r_{n})} }{\mu B^{cl}(x_{n}, t_n r_{n})} 
(y) 
\le
\sum_{ k_1 < n \le k_2}\frac{w_n \mathbf{1}_{B^{cl}(x_{n},  t_n r_{n})} }{\mu B^{cl}(x_{n}, t_n r_{n})} 
(\pi_1 (y)) ,
\end{equation*}
so
\begin{equation*}
\sum_{ k_1 < n \le k_2}\frac{w_n \mathbf{1}_{B^{cl}(x_{n},  t_n r_{n})} }{\mu B^{cl}(x_{n}, t_n r_{n})} 
= 
\mathbf{1}_A
\sum_{ k_1 < n \le k_2}\frac{w_n \mathbf{1}_{B^{cl}(x_{n},  t_n r_{n})} }{\mu B^{cl}(x_{n}, t_n r_{n})} 
+
\mathbf{1}_{\cup_{n \ge 0} A_n}
\sum_{ k_1 < n \le k_2}\frac{w_n \mathbf{1}_{B^{cl}(x_{n},  t_n r_{n})} }{\mu B^{cl}(x_{n}, t_n r_{n})} 
\end{equation*}
\begin{equation*}
\le 
2 \times
\mathbf{1}_A
\sum_{ k_1 < n \le k_2}\frac{w_n \mathbf{1}_{B^{cl}(x_{n},  t_n r_{n})} }{\mu B^{cl}(x_{n}, t_n r_{n})}. 
\end{equation*}
It follows that 
\begin{equation*}
II
\le
2 
\left\| \mathbf{1}_A \sum_{ k_1 < n \le k_2}\frac{w_n \mathbf{1}_{B^{cl}(x_{n},  t_n r_{n})} }{\mu B^{cl}(x_{n}, t_n r_{n})} \right\|_p.
\end{equation*}
Now for every measurable set $E \subset X$, $\mu E \le 2 \mu ( \pi_1 E)$, so $\mu$ is doubling
on $A$, with doubling constant $C \le 8$ (the bound from (\ref{double}), which is not tight
on $A$) so by Theorem \ref{M^uloc} once more,
\begin{equation}\label{boundforII}
II
\le
2  \frac{8^{3/q}q }{q-1} \left\| \sum_{1 \le n \le  N}\frac{w_n \mathbf{1}_{B^{cl}(x_{n}, r_{n})} }{\mu B^{cl}(x_{n}, r_{n})} \right\|_p,
\end{equation}
and thus, we can take $C_2 = 2 C_1.$

Finally, if the set of indices $n$ satisfying $n > k_2$ is not empty, for each
 such $n$ the ball $B^{cl}(x_{n}, t_n r_{n})$ has radius bounded by 1 and center contained in
some $A_N$, $N\ge 1$. It follows that if $j > k_2$ and $B^{cl}(x_{n}, t_n r_{n})
\cap B^{cl}(x_{j}, t_j r_{j}) \ne \emptyset$, then $x_j \in A_N$ also.
We need to bound
\begin{equation}\label{drei}
III:= \left\| \sum_{ k_2 < n \le N}\frac{w_n \mathbf{1}_{B^{cl}(x_{n},  t_n r_{n})} }{\mu B^{cl}(x_{n}, t_n r_{n})} \right\|_p.
\end{equation}

Let $\{A_{N_1}, \dots , A_{N_s}\}$ be precisely the set of $A_n$'s that intersect
(or equivalently, contain the center) of at least one of the balls  $B^{cl}(x_{n}, t_n r_{n})$, 
$n > k_2$. For each $1 \le j \le s$, let  $\mathcal{C}_j:= \{B^{cl}(x_{j_1},  t_{j_1}  r_{j_1}), \dots , 
B^{cl}(x_{j_{s (j)}},  t_{j_{s (j)}}  r_{j_{s (j)}} )\}$ be  the collection of all balls with radius bounded by
1 and with center in $A_{N_{j}}$. Set
$$
G_j (x) := 
 \sum_{ 1 \le  i\le s (j)}\frac{w_n \mathbf{1}_{B^{cl}(x_{j_{i}},  t_{j_{i}}  r_{j_{i}})}}{\mu B^{cl}(x_{j_{i}},  t_{j_{i}}  r_{j_{i}} )} (x)
\mbox{ \ \ and \ \ } g_j (x) := 
 \sum_{ 1 \le i \le s (j)}\frac{w_n \mathbf{1}_{B^{cl}(x_{j_{i}},   r_{j_{i}})}}{\mu B^{cl}(x_{j_{i}},    r_{j_{i}} )} (x).
$$
Since for $i\ne j$ the functions $G_i$ and $G_j$ have disjoint supports,  it is enough to show that 
for a suitable constant $C_3$, we have
$
III_j^p:= \int G_j^p d\mu 
\le
C_3^p
\int g_j^p d\mu. 
$

Again we break up each of these sums $G_j$  into different smaller sums 
(not necessarily disjoint, so some balls may appear more than once).

First of all, consider the collection $\mathcal{C}_{j,1}$ of balls
$ B^{cl}(x_{j_{i}},  t_{j_{i}}  r_{j_{i}})$ in $\mathcal{C}_{j}$ such that 
$$
\mu (B^{cl}(x_{j_{i}},  t_{j_{i}}  r_{j_{i}}) \cap A) \le \mu (B^{cl}(x_{j_{i}},  t_{j_{i}}  r_{j_{i}}) \cap 
A_{N_j }).
$$
 To avoid more subscripts, we rename the balls in $\mathcal{C}_{j,1}$ as  $U_1, \dots, U_m$, and we call
 $u_1, \dots, u_m$  the corresponding
contracted balls.  Note then that $\mu (u_j \cap A) \le \mu (u_j \cap 
A_{N_j })$, since either $u_j \cap A = \emptyset$, or if $u_j \cap A \ne  \emptyset$,
then $u_j \cap  A_{N_j } = A_{N_j }$, so 
$$
\mu (u_j \cap A) \le \mu (U_j \cap A) \le \mu (U_j \cap A_{N_j })
 = \mu (A_{N_j }) = \mu (u_j \cap A_{N_j }).
 $$
Since $\mu$ on 
 $A_{N_j }$ is just a multiple of the one dimensional  Lebesgue measure $\lambda$, we have 
$$
\left\|M^u_{\mu} \right\|_{L^q (A_{N_j }) \to L^q  (A_{N_j })} \le  \frac{2^{1/q}q }{q-1},  
$$ 
for, as is well known,  $\left\|M^u_{\lambda} \right\|_{L^1 (A_{N_j }) \to L^{1, \infty}  (A_{N_j })} \le  2 $.

\

Thus,  
 \begin{equation}\label{dreieinein}
 \left\| \sum_{1 \le n \le  m}\frac{w_n \mathbf{1}_{U_{n}   } }{\mu U_{n}  } \right\|_p
\le
 \left\| \sum_{1 \le n \le  m}\frac{2 w_n  \mathbf{1}_{(U_{n}  \cap A_{N_j } )  } }{\mu  (U_{n}  \cap A_{N_j } ) } \right\|_p
\le
 \frac{2^{1/q}q }{q-1}  
  \left\| \sum_{1 \le n \le  m}\frac{2 w_n  \mathbf{1}_{(u_{n}  \cap A_{N_j } )  } }{\mu  (u_{n}  \cap A_{N_j } ) } \right\|_p
 \end{equation}
 \begin{equation}\label{dreieineinzwei}
 \le
 \frac{2^{1/q}q }{q-1}  
 \left\| \sum_{1 \le n \le  m}\frac{ 4 w_n \mathbf{1}_{u_{n}  } }{\mu u_{n}  } \right\|_p
= 
 \frac{2^{1/q} 4 q }{q-1}  
 \left\| \sum_{1 \le n \le  m}\frac{w_n \mathbf{1}_{u_{n}  } }{\mu u_{n}  } \right\|_p.
\end{equation}

Let $C_4 =  \frac{2^{1/q} 4 q }{q-1} $. Next, consider the collection  $\mathcal{C}_{j,2}$ of balls
$ B^{cl}(x_{j_{i}},  t_{j_{i}}  r_{j_{i}})$ such that 
$$
\mu (B^{cl}(x_{j_{i}},  t_{j_{i}}  r_{j_{i}}) \cap A) \ge \mu (B^{cl}(x_{j_{i}},  t_{j_{i}}  r_{j_{i}}) \cap 
A_{N_j }).
$$
We rename the balls in $\mathcal{C}_{j,2}$, ordering them by decreasing radii, as  $V_1, \dots, V_\ell$, and we call
 $v_1, \dots, v_\ell$  the corresponding
contracted balls.  Now let $k$ be such that for $1 \le i \le  k$,    $\mu (v_i \cap A) \ge \mu (v_i \cap 
A_{N_j })$, while for $k < i \le \ell$, 
 $\mu (v_i \cap A) < \mu (v_i \cap 
A_{N_j })$. Of course, it might happen that $k <1$ (when  $\mu (v_i \cap A) < \mu (v_i \cap 
A_{N_j })$ for all balls)  or that $k = \ell$  (when  $\mu (v_i \cap A) \ge \mu (v_i \cap 
A_{N_j })$ for every $1\le i \le \ell$). In either case, we just have one fewer sum to consider.

Now, regarding the balls  $V_{1}, \dots , V_k$, the situation is essentially the same
as when bounding $II$:  
 \begin{equation}\label{dreieineinzwei}
 \left\| \sum_{1 \le n \le  k}\frac{w_n \mathbf{1}_{V_{n}   } }{\mu V_{n}  } \right\|_p
\le
 \left\| \sum_{1 \le n \le  k}\frac{2 w_n  \mathbf{1}_{(V_{n}  \cap A)  } }{\mu  (V_{n}  \cap A) } \right\|_p
\le
 C_2 
  \left\| \sum_{1 \le n \le  k}\frac{2 w_n  \mathbf{1}_{(v_{n}  \cap A)  } }{\mu  (v_{n}  \cap A) } \right\|_p
 \end{equation}
 \begin{equation}\label{dreieineindrei}
 \le
 C_2
 \left\| \sum_{1 \le n \le  k}\frac{ 4 w_n \mathbf{1}_{v_{n}  } }{\mu v_{n}  } \right\|_p
= 
4 C_2
 \left\| \sum_{1 \le n \le  k}\frac{w_n \mathbf{1}_{v_{n}  } }{\mu v_{n}  } \right\|_p.
\end{equation}

As for the balls $V_{k + 1}, \dots , V_\ell$, for each $i \in (k, \ell]$ choose $t_i \in (0, 1]$ so that 
 $\mu (t_i V_i \cap A) =  \mu (t_i V_i \cap A_{N_j })$. This is always possible to do, since $\mu$ is a continuous
measure, and whenever $t_i V_i \cap A  \ne \emptyset$, we have   $ t_i V_i \cap A_{N_j } =  V_i \cap A_{N_j } = A_{N_j }$.
But now  the problem has been reduced to the preceding two cases, successively considered:
 \begin{equation}\label{last}
 \left\| \sum_{1 \le n \le  k}\frac{w_n \mathbf{1}_{V_{n}   } }{\mu V_{n}  } \right\|_p
\le
4 C_2 \left\| \sum_{1 \le n \le  k}\frac{w_n \mathbf{1}_{t_n V_{n}   } }{\mu t_n V_{n}  } \right\|_p
\le
4 C_2 C_4  
 \left\| \sum_{1 \le n \le  k}\frac{w_n \mathbf{1}_{v_{n}  } }{\mu v_{n}  } \right\|_p.
\end{equation}
Writing $C_3 := C_4 + 4 C_2 +  4 C_2  C_4$ and using the fact that the $G_i$'s are disjointly supported functions 
(hence, so are
the $g_i$'s) we have
$$
\int \left(\sum_{j=1}^s G_j\right)^p d\mu 
=
\sum_{j=1}^s \int G_j^p d\mu 
\le
\sum_{j=1}^s C_3^p \int g_j^p d\mu  
$$
$$
=
C_3^p \int \left(\sum_{j=1}^s g_j\right)^p d\mu
\le
C_3^p \left\| \sum_{1 \le n \le  N}\frac{w_n \mathbf{1}_{B^{cl}(x_{n}, r_{n})} }{\mu B^{cl}(x_{n}, r_{n})} \right\|_p^p. 
$$
Putting together the estimates for $I$, $II$, and $III$, we obtain (\ref{first}) with $C = C_1 + C_2 + C_3$.
\end{proof}

\begin{remark} Note that if the $p$-Boman covering property holds for some $p > 1$, with constant $C_p$, 
then it holds for
all the intermediate spaces $L^s(\mu)$, $1 \le s \le p$, with constant bounded by $C_p^{\frac{(s - 1) p}{s(p - 1)}}$.
This follows from
 Riesz-Thorin, applied to the operators defined by $T G = g$ on the cone of nonnegative functions of the form
$G= \sum_{1 \le n \le  N}\frac{w_n \mathbf{1}_{B^{cl}(x_{n}, r_{n})} }{\mu B^{cl}(x_{n}, r_{n})},
$ with $g = \sum_{1 \le n \le  N}\frac{w_n \mathbf{1}_{B^{cl}(x_{n},t_n  r_{n})} }{\mu B^{cl}(x_{n}, t_n  r_{n})}$
a contracted version of $G$ (so $0 < t_i \le 1$, $0 < w_i$).
\end{remark}


\begin{thebibliography}{WWW}



\bibitem[Al1]{Al1}  Aldaz, J. M.  {\em Local comparability of measures, averaging and  maximal averaging operators.}  Potential Anal. 49 (2018), no. 2, 309--330. Available at the Mathematics ArXiv. 

\bibitem[Al2]{Al2}  Aldaz, J. M.  {\em The Stein-Str\"omberg covering theorem in metric spaces.}  J. Math. Anal. Appl. 449 (2017), no. 2, 1741--1753.

\bibitem[Al3]{Al3}  Aldaz, J. M.  {\em On the pointwise domination of a function 
by its maximal function.}  To appear, Arch. Math., available at the Mathematics ArXiv. 

\bibitem[Boj]{Boj}  Bojarski, B. {Remarks on Sobolev imbedding inequalities.}  Complex analysis, Joensuu 1987, 52--68, Lecture Notes in Math., 1351, Springer, Berlin, 1988.


\bibitem[Bom]{Bom} Boman, J. {\em $ L^p$-estimates for very strongly elliptic systems,}  Report no. 29,
Department of Mathematics, University of Stockholm, Sweden (1982).

\bibitem[Bu]{Bu} Buckley, Stephen M. {\em Strong doubling conditions.} Math. Inequal. Appl. 1 (1998), no. 4, 533--542.


\bibitem[Ch]{Ch} Chua, Seng-Kee {\em Weighted Sobolev inequalities on domains satisfying the chain condition.} Proc. Amer. Math. Soc. 117 (1993), no. 2, 449--457. 

\bibitem[Ch1]{Ch1}  Chua, Seng-Kee {\em Sobolev interpolation inequalities on generalized John domains.} Pacific J. Math. 242 (2009), no. 2, 215--258. 

\bibitem[ChWh]{ChWh}  Chua, Seng-Kee; Wheeden, Richard L. {\em Self-improving properties of inequalities of Poincaré type on measure spaces and applications.} J. Funct. Anal. 255 (2008), no. 11, 2977--3007.




\bibitem[DeGuMa]{DeGuMa}  Deleaval, L., Gu\'edon, O. and Maurey, B. {\em Dimension free bounds for the Hardy-Littlewood maximal operator
associated to convex sets.}  To appear, Ann. Fac. Sci. Toulouse Math., available at the Math. ArXiv.


\bibitem[Gra]{Gra} L. Grafakos,  {\em Classical Fourier analysis. Third edition.}
Graduate Texts in Mathematics, 249. Springer, New York, (2014).

\bibitem[Ha]{Ha}  Hayes, C. A. {\em Derivation of the integrals of  $L^{(q)}$-functions.} 
Pacific J. Math. 64
(1976), 173--180.



\bibitem[HeKo]{HeKo}  Heinonen, Juha; Koskela, Pekka {\em Definitions of quasiconformality.} Invent. Math. 120 (1995), no. 1, 61--79.

\bibitem[HKST]{HKST}  J. Heinonen, P.  Koskela, N.  Shanmugalingam, J. T Tyson, 
{\em Sobolev spaces on metric measure spaces. 
An approach based on upper gradients.} New Mathematical Monographs, 27. Cambridge University Press, Cambridge, 2015.  




\bibitem[Hy]{Hy}  Hyt\"onen, Tuomas {\em A framework for non-homogeneous analysis on metric spaces, and the RBMO space of Tolsa.} Publ. Mat. 54 (2010), no. 2, 485--504. 

\bibitem[IhVa]{IhVa} Ihnatsyeva, Lizaveta; V\"ah\"akangas, Antti V. {\em Characterization of traces of smooth functions on Ahlfors regular sets.} J. Funct. Anal. 265 (2013), no. 9, 1870--1915.

\bibitem[Io]{Io} Ionescu, Alexandru D. {\em A  maximal operator and a covering lemma on non-compact symmetric spaces.}
Math. Res. Lett. 7 (2000), 83--93.

\bibitem[IwNo]{IwNo} Iwaniec, T.; Nolder, C. A. {\em Hardy-Littlewood inequality for quasiregular mappings in certain domains in $\mathbb{R}^n$.} Ann. Acad. Sci. Fenn. Ser. A I Math. 10 (1985), 267--282. 

\bibitem[KoShSt]{KoShSt}  Kozlovski, O.; Shen, W.; van Strien, S. {\em Rigidity for real polynomials.} 
Ann. of Math. (2) 165 (2007), no. 3, 749--841. 
 
   
\bibitem[Str]{Str} Str\"omberg, Jan-Olov
{\em Weak type $ L^1$ estimates for maximal functions on noncompact symmetric spaces.}
Ann. of Math. (2) 114 (1981), no. 1, 115--126. 

\end{thebibliography}
\end{document}